\documentclass[11pt]{amsart}

\usepackage[margin=1.2in]{geometry}
\usepackage{hyperref}
\usepackage{amsfonts}
\usepackage{amsmath}
\usepackage{amsthm}
\usepackage{amssymb}
\usepackage{float}
\usepackage{enumerate}
\usepackage{graphicx}
\usepackage{mathrsfs}
\usepackage{xcolor}
\usepackage[percent]{overpic}

\newtheorem{theorem}{Theorem}[section]

\newtheorem{prop}[theorem]{Proposition}

\newtheorem{thmintro}{Theorem}
\theoremstyle{definition}
\newtheorem{rmk}[theorem]{Remark}

\usepackage{enumitem}

\DeclareMathOperator{\inter}{\it{i}}

\DeclareMathOperator{\arccosh}{arccosh}

\DeclareMathOperator{\area}{area}

\title{Filling sets of curves on punctured surfaces}
\author{Federica Fanoni}
\address[Federica Fanoni]{Mathematics Institute, University of Warwick, UK}
\email{F.Fanoni@warwick.ac.uk}

\author{Hugo Parlier}
\address[Hugo Parlier]{Dept. Math., Hunter College CUNY, USA \& University of Fribourg, Switzerland}

\email{hugo.parlier@gmail.com}
\date{\today}
\thanks{Research supported by Swiss National Science Foundation grant number PP00P2\_128557 and P2FRP2\_161723. Both authors acknowledge support from U.S. National Science Foundation grants DMS 1107452, 1107263, 1107367 ``RNMS: GEometric structures And Representation varieties'' (the GEAR Network).}

\begin{document}
\begin{abstract}
We study filling sets of simple closed curves on punctured surfaces. In particular we study lower bounds on the cardinality of sets of curves that fill and that pairwise intersect at most $k$ times on surfaces with given genus and number of punctures. We are able to establish orders of growth for even $k$ and show that for odd $k$ the orders of growth behave differently. We also study the corresponding questions when one requires that the curves be represented as systoles on hyperbolic complete finite area surfaces. 
\end{abstract}
\maketitle
\section{Introduction}

A set of simple closed curves on a surface is said to fill if it cuts the surface into topological disks and once-punctured disks. Any such filling set must contain at least two curves; by a simple topological argument (see for instance \cite{ts}) if two curves fill, they must intersect at least $|\chi|$ times, where $\chi$ is the Euler characteristic of the surface. So if we bound the number of times they can intersect and increase the complexity of the surface, we will need more curves; but how many? The main goal of this paper is to give an answer to this question.

For closed surfaces of genus $g$, it is known \cite{app} that the number $N$ of curves in a filling set of curves that pairwise intersect at most $k$ times satisfies
$$N^2-N\geq\frac{4g-2}{k}.$$
Moreover the bound is essentially sharp.

In this paper we study finite type surfaces of negative Euler characteristic and with punctures. Somewhat surprisingly, the bounds we obtain differ depending on the parity of $k$, the number of times curves are allowed to pairwise intersect. For even $k$, we obtain a similar result to the one for closed surfaces mentioned above. 
\begin{thmintro}\label{thmintro:even}
Let $S$ be a surface with at least one puncture and let $k$ be a positive even integer. Any filling set of curves on $S$ pairwise intersecting at most $k$ times has cardinality at least $N$, where $N$ is the smallest integer satisfying
$$N(N-1)\geq \frac{2}{k}|\chi(S)|.$$
Furthermore, if $g(S)\leq 1$ then there exists a filling set of curves pairwise intersecting at most $k$ times of size $N$. However if the surface has genus at least two then there exists a filling set of size less than 
$$
\sqrt{\frac{2|\chi(S)|}{k} +\frac{1}{4} }+ 6
$$
\end{thmintro}
Note that the above formulas determine the order of growth (in function of the Euler characteristic) of a minimal filling set (the leading term being $\sqrt{\frac{2|\chi(S)|}{k}}$). 

In contrast for odd $k$ we show that the order of growth is different; to explain our result we need to talk a little bit about a related problem. We denote $M_g(k)$ the {\it maximum} number of curves that pairwise intersect at most $k$ times on a closed genus $g$ surface.  Determining $M_g(k)$ is a surprisingly hard problem. Although bounds are known (see the work of Przytycki in \cite{przytycki}  and also Aougab in \cite{aougab}, Juvan--Malni\v{c}--Mohar in \cite{jmm}, and Malestein--Rivin--Theran \cite{mrt}) even the rough order of growth of $M_g(k)$ is not known. For $k=1$, its growth in terms of the genus is known to be somewhere between quadratic and cubic. This somewhat mysterious quantity appears in our next theorem. 

\begin{thmintro}\label{thmintro:odd} Let $S$ be a surface of genus $g$ with at least one puncture. Then a filling set of curves pairwise intersecting at most $k$ times has cardinality at least $N$, where $N$ is the smallest integer satisfying
$$\frac{k}{2}N(N-1)-\frac{N}{2}\left(\frac{N}{M_g(k)}-1\right)\geq|\chi(S)|.$$
\end{thmintro}

Note that the order of growth is really different from the one in Theorem \ref{thmintro:odd} because of the extra term $-\frac{N}{2}\left(\frac{N}{M_g(k)}-1\right)$. 

In light of the problem of determining and realizing the quantity $M_g(k)$, when $k$ is odd we are not able to give explicit constructions for small filling sets of curves pairwise intersecting at most $k$ times that match our lower bound, with one notable exception. 

When $g=1$ and $k=1$, it is not difficult to show that $M_g(1) = 3$. Using this, when the surface is a punctured torus and $k=1$ we can prove a precise result.
\begin{thmintro}\label{thmintro:torus}
Let $ \delta_1,\hdots,  \delta_N$ be a filling set of curves that pairwise intersect at most once on a torus with $n$ punctures. Then 
$$
N \geq \sqrt{3n}.
$$
Conversely for a torus with $n$ punctures, there exists a filling set of $N$ isotopically distinct simple curves $ \delta_1,\hdots, \delta_N$ that pairwise intersect at most once with 
$$
N \leq \sqrt{3n+1}.
$$
\end{thmintro}

The results described up until now are purely topological. One motivation for understanding the topology of curves that pairwise intersect at most a small number of times comes from the study of {\it systoles} on surfaces. A systole is a shortest closed non contractible curve non isotopic to boundary. For any given Riemannian metric, systoles intersect at most twice (and at most once on a closed surface). An important class of metrics is complete finite area hyperbolic surfaces; we consider systoles on these.  

In particular, we can ask what happens to our previous bounds if one requires that the curves be systoles and we show that the growth of the lower bound is very different.
\begin{thmintro}\label{thmintro:systoles}
Let $S$ be a hyperbolic surface of signature $(g,n)$ and systole length $\ell$. If $\gamma_1,\dots,\gamma_M$ is a filling set of systoles, then
$$M\geq \frac{2\pi(2g-1)+\pi(n-2)}{4\ell}.$$
\end{thmintro}
We also give examples of constructions of surfaces with a filling set of systoles of cardinality linear in the Euler characteristic, showing that the order of growth of Theorem \ref{thmintro:systoles} is roughly correct.

The paper is organized as follows. In Section \ref{prelim} we give the main definitions and prove Theorem \ref{thmintro:torus}. In the subsequent section, we prove Theorem \ref{thmintro:even}, treating separately the constructions for the case of spheres, tori and higher genus surfaces. Theorem \ref{thmintro:odd} is proven in Section \ref{odd} and the final section is dedicated to the lower bounds on the number of filling systoles.

\section{The case of genus $1$}\label{prelim}

In this section we introduce some of the objects of interest and illustrate our objective by proving Theorem \ref{thmintro:torus}. Although it is relatively straightforward, it contains many of the main steps that will be used in the sequel.

A simple closed curve is {\it essential} if it is not homotopic to a point or to a puncture. Throughout the paper, by {\it curve} we will mean a simple, closed, essential curve.

A set $\Gamma$ of pairwise non-homotopic curves on a surface $S$ {\it fills} if the complement $S\setminus \Gamma$ is a union of disks and once-punctured disks. A {\it $k$-filling set} is a set of pairwise non-homotopic curves which fill and pairwise intersect at most $k$ times.

With this notation, Theorem \ref{thmintro:torus} can be restated as follows.

\begin{theorem}\label{thm:torus}
Let $ \delta_1,\hdots,  \delta_N$ be a $1$-filling set on a torus $T$ with $n$ punctures. Then 
$$
N \geq \sqrt{3n}.
$$
Conversely for $T$ a torus with $n$ punctures, there exists a $1$-filling set of cardinality $N$ with 
$$
N \leq \sqrt{3n+1}.
$$
\end{theorem}

\begin{proof}
We begin by recalling the well-known fact that for $n=0$ (or $n=1$), there can be at most $3$ topologically distinct curves that pairwise intersect at most once. Associated to $T$ is the torus $T^{0}$ obtained by forgetting the punctures of $T$. This map acts on curves of $T$ sending them to curves on $T^{0}$ and is called the forgetful map. Note that if two curves on $S$, say $\delta$ and $\tilde{\delta}$, intersect at most once, their images on $S$ do as well. 

Now given a set of curves $ \delta_1,\hdots,  \delta_N$ that pairwise intersect at most once, let's consider the curves obtained on $T^{0}$ via the forgetful map. The image consists in at most three curves. Let's denote these curves $\alpha, \beta$ and $\gamma$ (if the image is smaller, we arbitrarily choose the remaining curves so that they intersect at most once). We can split the the curves $ \delta_1,\hdots,  \delta_N$ into three sets depending on whether they are preimages of $\alpha$, $\beta$ or $\gamma$. Up to renumbering, let's assume that the preimages of $\alpha$ are 
$$
\delta_1,\hdots,\delta_a,
$$
those of $\beta$ are 
$$
\delta_{a+1},\hdots,\delta_{a+b},
$$
and those of $\gamma$ are 
$$
\delta_{a+b+1},\hdots,\delta_{a+b+c}
$$
where $N=a+b+c$. 

Observe that $\inter(\delta_i,\delta_j) = 1$ if and only if $\delta_i$ and $\delta_j$ are the preimages of different curves among the set $\alpha,\beta,\gamma$. As such the total number of pairwise intersections among the curves $\delta_1,\hdots,\delta_N$ is 
$$
a b + bc + a c.
$$
We can assume no intersection points coincide on $S$; by an Euler characteristic argument the number above is also the number of connected components of $T \setminus \{\delta_1,\hdots,\delta_N\}$. As there must be at least as many connected components as punctures we obtain:
$$
n \leq a b + bc + a c.
$$
Via Lagrange, the quantity $ab + bc + ac$ is maximal among all $a,b,c$ satisfying $a+b+c=N$ when $a,b,c$ are equal. Thus
$$
n\leq \frac{1}{3} N^2
$$
which proves the first assertion.

Note that if $N$ is not divisible by $3$ we can get a slightly better bound. Indeed, in this case the maximum that can be achieved is $\frac{N^2-1}{3}$ (for instance for $a=\lfloor\frac{N}{3}\rfloor$, $b=\lfloor\frac{N}{3}\rfloor+1$ and $c=a$ or $b$, depending on whether $N\equiv 1$ or $2$ modulo $3$). So in this case we get
$$n\leq \frac{N^2-1}{3}.$$

To prove the second assertion it suffices to reverse engineer the above process. Consider a torus $T$ with three curves $\alpha,\beta$ and $\gamma$ which all pairwise intersect at most once. We begin by choosing the minimal $N$ satisfying the above inequality and take $a:=\lfloor{\frac{N}{3}\rfloor}$ parallel copies of $\alpha$, $b:= \lfloor{\frac{N}{3}\rfloor} + d$ parallel copies of $\beta$ and $c:= \lfloor{\frac{N}{3}\rfloor} + d'$ parallel copies of $\gamma$ where $0\leq d,d'\leq 1$ are integers and $a+b+c =N$. We now have a collection of $N$ curves on $S$.

Now as above, the number of connected components of the complementary regions to all of the curves is $a b + bc + a c$. We place at most one puncture in each of the connected regions for a total of $n$ punctures. The result is an $n$-times punctured torus with a filling set of curves that satisfies the desired inequality. 
\end{proof}
\section{The topological setup: case $k$ even}\label{even}
In this section we will always assume $k$ to be an even positive integer.

We begin by proving a lower bound on the number of curves in a $k$-filling set of curves on any punctured surface.
\begin{theorem}\label{thm:lower_even}
Let $S$ be a surface with at least one puncture and of negative Euler characteristic. Any $k$-filling set of curves on $S$ has cardinality at least $N$, where $N$ is the smallest integer satisfying
$$N(N-1)\geq \frac{2}{k}|\chi(S)|.$$
\end{theorem}

\begin{proof}
Suppose $\{\gamma_1,\dots,\gamma_m\}$ is a $k$-filling set of curves on a surface $S$ of signature $(g,n)$. Up to homotopy, we can assume no three curves intersect in the same point and each intersection is transversal. This implies that the curves define a $4$-valent graph $G$ on $S$, with the intersections as vertices and edges given by the arcs of the curves. Denote by $v(G)$ and $e(G)$ the number of vertices and edges of $G$, and by $f(G)$ the number of connected components of $S\setminus G$. By the hand-shaking lemma, since $G$ is $4$-valent, we have
$$e(G)=2v(G).$$
Any two curves pairwise intersect at most $k$ times, so the number of vertices satisfies
$$v(G)\leq k {m\choose 2}=k\frac{m(m-1)}{2}.$$
Moreover, since the set of curves is filling and the surface has $n$ punctures, there are at least $n$ connected components of $S\setminus G$, i.e.\ $f(G)\geq n$. By computing the Euler characteristic of $S$ as $v(G)-e(G)+f(G)$ and using the estimates on $v(G)$ and $f(G)$ we obtain the desired lower bound.
\end{proof}

\begin{rmk}
Note that the lower bound of Theorem \ref{thm:lower_even} holds for odd $k$ as well, but, as we will show later, for $k$ odd we can get a better bound.
\end{rmk}

We begin with the case of the sphere - in this case we can show that the lower bound of Theorem \ref{thm:lower_even} is sharp.
\begin{theorem}\label{thm:spheres}
Let $S$ be a sphere with $n\geq 4$ punctures. There exists a  $k$-filling set of curves on $S$ of cardinality $N$, where $N$ is the smallest integer satisfying
$$N(N-1)\geq \frac{2n-4}{k}.$$
\end{theorem}

\begin{proof}
%
Fix $k$; we start by constructing the set of curves in the case in which
$$n=\frac{k N(N-1)+4}{2}$$
for some integer $N$.

Consider the rectangle $[0,k\pi]\times [-1,1]\subseteq \mathbb{R}^2$ and the graphs of the functions $f_s(x)=\sin(x+s\varepsilon)$ for $s\in\{0,1,\dots ,N-1\}$ and $\varepsilon$ small. Note that we can choose $\varepsilon$ small enough so that any two of the above graphs intersect exactly $k$ times and there are no triple intersections (as in Figure \ref{fig:sine}).
\begin{figure}[H]
\includegraphics{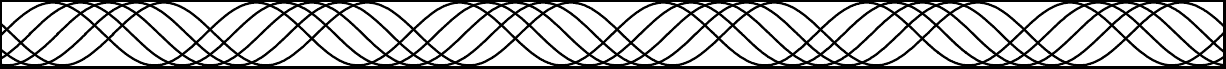}
\caption{The graphs of $f_0, f_1,f_2$ and $f_3$ on the rectangle $[0,12\pi]\times[-1,1]$}\label{fig:sine}
\end{figure}
 Consider the cylinder obtained by identifying $(0,t)$ with $(k\pi,t)$ for any $t\in[-1,1]$. On this cylinder, the graphs project to $N$ curves, all pairwise intersecting exactly $k$ times, with no three curves intersecting in the same point. We glue disks to the two boundary components of the cylinder to obtain a sphere. As in the proof of the lower bound, we consider the graph $G$ induced by the curves on the sphere. Again it is $4$-valent, so $e(G)=2v(G)$. Since all curves pairwise intersect exactly $k$-times and no three curves have a common intersection, we have
$$v(G)=k {N\choose 2}=k\frac{N(N-1)}{2}=n-2.$$
Since $2=\chi(\mathbb{S}^2)=v(G)-e(G)+f(G)$, the number of connected components of the complement of $G$ is
$$f(G)=2+v(G)=n.$$
So we add a puncture to each connected component. This gives a $n$-punctured sphere with a $k$-filling set of the desired size.

Now consider $n$ not of the form $\frac{k N(N-1)+4}{2}$ for any $N$. Then there exists an integer $N$ such that
\begin{equation}\label{eqn:boundN}
\frac{k (N-1)(N-2)+4}{2}<n<\frac{k N(N-1)+4}{2}.
\end{equation}
We construct a sphere with $N$ curves pairwise intersecting $k$ times as in the previous case. The difference is that this time we have less cusps than connected components. To be sure that no two curves are homotopic, it is enough to place a single puncture to separate the first curve from all of the other curves, then a puncture between the second and the subsequent curves and so on. Hence it is enough to have $n\geq N-1$ punctures and this inequality holds via the lower bound in \ref{eqn:boundN}. So again we obtain a filling set of curves of the right size.
\end{proof}

With the same techniques we can prove a similar statement for tori.

\begin{theorem}
Let $T$ be a torus with $n\geq 1$ punctures and $k$ be even. There exists a $k$-filling set of curves on $T$ of cardinality $N$, where $N$ is the smallest integer satisfying
$$N(N-1)\geq \frac{2n}{k}.$$
\end{theorem}
\begin{proof}
The proof is essentially the same as for spheres. The only difference is that instead of gluing two disks to turn the cylinder with the curves into a sphere, we glue its two boundary components to get a torus.
\end{proof}
To prove the result in the case of surfaces of genus at least two we combine the idea of the construction in the cases of spheres and tori and a known result about $k$-filling sets on closed surfaces from \cite{app}.
\begin{theorem}\label{thm:genus}
Let $S$ be a surface of signature $(g,n)$, with $g\geq 2$ and $n\geq 1$. For any even $k\geq 2$, there is a $k$-filling set on $S$ of size $N$ satisfying
$$\frac{5}{2}+\sqrt{\frac{1}{4}+\frac{2|\chi(S)|}{k}}\leq N<6+\sqrt{\frac{1}{4}+\frac{2|\chi(S)|}{k}}.$$
\end{theorem}

\begin{proof} We construct a $k$-filling set of the desired size.\\
Consider the closed surface $S^0$ obtained by filling in the punctures. By a result in \cite{app}, we know that there exists an $k$-filling set $\mathcal{C}^0$ of cardinality $x$ or $x+1$, where $x$ is the smallest integer satisfying
\begin{equation}\label{eqn:boundx}
x(x-1)\geq \frac{4g-2}{k}.
\end{equation}
The construction in \cite{app} has most curves that pairwise intersect exactly $k$ times. In fact, the curves are constructed algorithmically and this property is true for all but (possibly) the final two curves in the construction. Pick a curve $\gamma$ from the construction that is not one of the final two and replace it with a thin cylinder with a set of $y+1$ curves (as in the construction of Theorem \ref{thm:spheres}). We obtain set of curves $\mathcal{C}$ of cardinality $x+y$ or $x+y+1$. Choose $y$ to be the smallest integer such that $S^0\setminus \mathcal{C}$ has at least $n$ connected components. Since at least $x+y-2$ curves of $\mathcal{C}$ pairwise intersect exactly $k$ times, the number of components of $S^0\setminus \mathcal{C}$ is at least
$$2-2g+k{x+y-2\choose 2}.$$
So we choose $y$ to be the smallest integer such that
\begin{equation}\label{eqn:boundy}
2-2g+k{x+y-2\choose 2}\geq n.
\end{equation}
Using the two inequalities \ref{eqn:boundx} and \ref{eqn:boundy}, one can compute that $x+y$ satisfies
$$\frac{5}{2}+\sqrt{\frac{1}{4}+\frac{2|\chi(S)|}{k}}\leq x+y<5+\sqrt{\frac{1}{4}+\frac{2|\chi(S)|}{k}}.$$
Set $N=|\mathcal{C}|$; we know that $N=x+y$ or $x+y+1$. Moreover, the complement of  $\mathcal{C}$ is a union of disks. We want to place at most one puncture per connected component. By construction, we have enough components. Also, all curves are pairwise non-homotopic, except possibly for the $y+1$ in the thin cylinder. To be sure these are pairwise non-homotopic, it is enough to have at least $y$ punctures, i.e.\ it is enough to have $n\geq y$. This is true if $y=0$ or $y=1$ (by assumption). If we have only two punctures, it is enough to have two curves in the cylinder, so $y\leq 1$. This means that if $y=2$ or $y=3$, we must have had $n\geq 3$. So if $y\leq 3$, the condition $n\geq y$ is satisfied.

Assume now $y\geq 4$. By the minimality of $y$, we know that
$$2-2g+k{x+y-3\choose 2}< n$$
which implies, using inequality \ref{eqn:boundx} and basic computations
$$n\geq 2+\frac{k}{2}(y-3)^2+\frac{k}{2}(y-3).$$
So it is enough to have
$$2+\frac{k}{2}(y-3)^2+\frac{k}{2}(y-3)\geq y,$$
which holds, under our assumption $y\geq 4$.

Thus we can add punctures in chosen connected components and we obtain a filling set of size $N$.
\end{proof}

\section{The topological setup: case $k$ odd}\label{odd}
In this section we will always assume $k$ to be an odd positive integer.

A {\it $k$-system} on a surface $S$ is a set of curves which pairwise intersect at most $k$-times. We set $M_g(k)$ to be the maximum cardinality of a $k$-system on a closed surface of genus $g$.

\begin{theorem}
Let $S$ be a surface of signature $(g,n)$, with $g\geq 1$ and $n\geq 1$. Then a $k$-filling set has cardinality at least $N$, where $N$ is the smallest integer satisfying
$$\frac{k}{2}N(N-1)-\frac{N}{2}\left(\frac{N}{M_g(k)}-1\right)\geq|\chi(S)|.$$
\end{theorem}

\begin{proof}
Let $\Gamma=\{\gamma_1,\dots ,\gamma_N\}$ be a $k$-filling set on $S$; up to isotopy we can assume that there are no triple intersection points and all intersections are transverse. As in the proof of the lower bound in Theorem \ref{thm:torus}, we consider the associated surface $S^0$ obtained by forgetting the punctures and the forgetful map $\pi:S\rightarrow S^0$. Let $\delta_1,\dots,\delta_M$ be the isotopy classes in $\pi(\Gamma)$ and consider the families $\mathcal{F}_i=\pi^{-1}(\delta_i)$. Note that if two curves in $\Gamma$ belong to the same family $\mathcal{F}_i$, they are isotopic on $S^0$, so they can only have an even number of intersections. Since $k$ is odd, this means that they intersect at most $k-1$ times. Let $a_i$ be the cardinality of $\mathcal{F}_i$.

As in the proof of Theorem \ref{thm:lower_even}, we consider the graph $G$ induced  by $\Gamma$ on $S^0$. Again it is $4$-valent, thus $e(G)=2v(G)$ and $f(G)=\chi(S^0)+v(G)\geq n$. We have
$$
v(G)=\underbrace{\sum_{i=1}^M\left(\sum_{\alpha,\beta\in \mathcal{F}_i}|\alpha\cap\beta|\right)}_{\stackrel{\mbox{\small intersections between curves}}{\mbox{\small in the same family}}}+\underbrace{\sum_{i<j}\left(\sum_{\alpha\in\mathcal{F}_i,\beta\in \mathcal{F}_j}|\alpha\cap\beta|\right)}_{\stackrel{\mbox{\small intersections between curves}}{\mbox{\small in different families}}}
$$
By what we said before, the intersections $|\alpha\cap\beta|$ in the first sum are bounded by $k-1$ and the ones in the second sum simply by $k$. So
$$
v(G)\leq\sum_{i=1}^M\left(\sum_{\alpha,\beta\in \mathcal{F}_i}(k-1)\right)+\sum_{i<j}\left(\sum_{\alpha\in\mathcal{F}_i,\beta\in \mathcal{F}_j}k\right)
=(k-1)\sum_{i=1}^M{a_i\choose 2}+k\sum_{i<j}a_ia_j$$
By Lagrange, $(k-1)\sum_{i=1}^M{a_i\choose 2}+k\sum_{i<j}a_ia_j$ is maximized for $a_1=\dots a_M=\frac{N}{M}$. Using this and the fact that $M\leq M_g(k)$ we get
$$v(G)\leq \frac{k}{2}N(N-1)-\frac{N}{2}\left(\frac{N}{M_g(k)}-1\right).$$
Combining this estimate with $\chi(S^0)+v(G)\geq n$ we obtain our claim.
\end{proof}

\section{Systoles of punctured surfaces}\label{systoles}

In this section we prove bounds on how the minimum number of filling systoles grows in function of the number of punctures of a (finite area complete) hyperbolic surface. Our bounds will show that, like in the case of closed surfaces (see \cite{app}), the topological condition of intersecting at most once or twice is very far from the geometric condition of being systoles. 

\begin{theorem}\label{thm:systoles}
Let $S$ be a hyperbolic surface of signature $(g,n)$ and systole length $\ell$. If $\gamma_1,\dots,\gamma_M$ is a filling set of systoles, then
$$M\geq \frac{2\pi(2g-1)+\pi(n-2)}{4\ell}.$$
\end{theorem}

\begin{rmk} Theorem \ref{thm:systoles}, together with the systole bounds in \cite{schmutz} and \cite{fp}, implies that if $g$ is fixed and $n$ goes to infinity, then $M\geq An$, for some constant $A$. If $n$ is fixed and $n$ goes to infinity, $M\geq B\frac{g}{\log g}$, for some constant $B$.
\end{rmk}

\begin{proof}
The main idea is to use the isoperimetric inequality of the hyperbolic plane. 

Consider a hyperbolic surface $S$ with $n$ punctures and its set of filling systoles $\gamma_1,\hdots,\gamma_M$ of length $\ell$. We begin by considering the unique hyperbolic metric $S'$ with cone angle $\pi$ in every puncture conformally equivalent to $S$ outside of the cone angle points. This surface is uniquely determined by the conformal structure of $S$ (see Troyanov \cite{troyanov}) and via the Pick-Schwartz inequality (see Troyanov \cite{troyanov2}) enjoys a certain number of properties. All closed curves on $S'$ are of length strictly less than the corresponding curves of $S$ and in particular
\begin{equation}\label{eqn:S'}
\ell_{S'}(\gamma_k) < \ell_{S}(\gamma_k)=\ell
\end{equation}
for all $k=1,\hdots,M$.

Because the curves $\gamma_k$, $k=1,\hdots,M$ fill $S$, they also fill $S'$. Cutting along the curves then produces a collection of polygons, each with at most one cone point in its interior. We want to apply the isoperimetric inequality of the hyperbolic plane to this set -- but the cone points are an obstruction. 

To get rid of this obstruction we perform the following covering operation on those with a cone point of angle $\pi$: a polygon with a cone point of angle $\pi$ is the quotient of a centrally symmetric polygon by an involution so we replace is by its double cover which is a genuine hyperbolic polygon. We now have a full collection of hyperbolic polygons $P_1,\hdots,P_p$. 

By the isoperimetric inequality the boundary lengths of the polygons satisfy 
$$
\sum_{k=1}^{p} \ell(\partial P_k) > \sqrt{\area(S')^2+4\pi\area(S')}>\area(S')=2\pi(2g-1)+\pi(n-2)
$$
We'll now look at how the sum above relates to the sum of the $\ell_S(\gamma_k)$s. Using inequality \ref{eqn:S'}, fact that each $\gamma_k$ contributes exactly twice to the length of the filling set and finally the fact that the length of a $\partial P_j$ may have been doubled, we have:

$$
\sum_{j=1}^{p} \ell(\partial P_j) \leq 4 \sum_{k=1}^{M} \ell_S (\gamma_k)=4M\ell.
$$
Putting the two inequalities above together gives the result. 
\end{proof}

Actually, the growth of the lower bound in Theorem \ref{thm:systoles} is roughly correct. Indeed, we can construct families of surfaces with a filling set of systole growing linearly in $g+n$.

The first example is the family of surfaces $\{S_{g,n(g)}\}_{g\geq 2}$ constructed in Lemma $3.5$ of \cite{fp}. For every $g\geq 2$, $S_{g,n(g)}$ has genus $g$, $n(g)=46(g-1)$ cusps and an ideal triangulation where all but one vertex have degree $6$ and the remaining vertex has degree $12g-6$. Systoles correspond to edges between two vertices of degree $6$, so one can show that there are $36g-54$ systoles. As these correspond to all edges of the triangulation, except the ones incident to a single vertex, they fill. Moreover, an explicit computation shows that the length of a systole is precisely $\arccosh 3$ (and thus independent of $g$). Note that Theorem \ref{thm:systoles} gives, for these surfaces, that a set of filling systoles on $S_{g,n(g)}$ must have at least
$$\frac{25\pi}{2\arccosh 3}(g-1)$$
so the construction gives surfaces with less than twice the necessary amount of curves. 

The second example is a family of spheres with a filling set of systoles of cardinality equal to the set of punctures.
\begin{prop}
For any $n\geq 4$, there is a $n$-punctured sphere with a set of filling systoles of cardinality $n$.
\end{prop}
\begin{proof}
Consider an ideal maximally symmetric $n$-gon in the hyperbolic plane. In the Poincar\'e disk model, we can think of it as the $n$-gon with ideal vertices $v_k=e^{i\frac{2\pi k}{n}}$, for $k$ from $0$ to $n-1$. Take two copies of it and glue them such that the endpoints of orthogonals between two non-consecutive sides are identified. In particular this means that the orthogonals give simple closed geodesics on the sphere. We will show that these are the only systoles. Since there are $n$ of these curves and they fill the surface, this concludes the proof.

Consider the center of the polygon (in the Poincar\'e disk model this is the origin). To compute its distance $d$ to any side, consider the right-angled triangle given by the orthogonal from the center to a side, the geodesic from the center to one of the two vertices of the side and the part of the side from the vertex to the foot of the orthogonal. By hyperbolic trigonometry, the distance $d$ satisfies
$$\cosh d=\frac{1}{\sin\frac{\pi}{n}}.$$

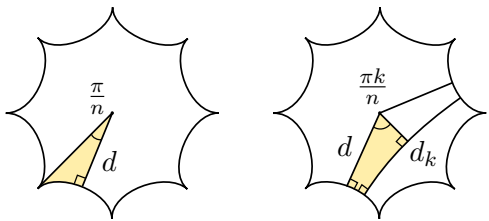
\begin{figure}[H]
\begin{center}
\begin{overpic}[scale=1]{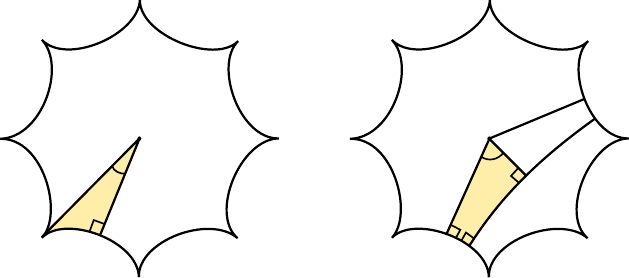}
\put (20,10) {$d$}
\put (17,24) {$\frac{\pi}{n}$}
\put (69,13) {$d$}
\put (84,12) {$d_k$}
\put (73,26) {$\frac{\pi k}{n}$}
\end{overpic}
\end{center}
\caption{Computing $d$ and $d_k$}
\end{figure}

Consider now two non-consecutive sides $a$ and $b$. Suppose $k-1$ is the minimum number of sides between them. Then the smallest angle between the two orthogonals from the center to $a$ and $b$ is $\frac{2\pi k}{n}$. These two orthogonals and the common orthogonal $d_k$ between $a$ and $b$ determine a pentagon with four right angles and and a $\frac{2\pi k}{n}$ angle. By taking the orthogonal from the center to $d_k$, we cut the pentagon into two quadrilaterals with three right angles and a $\frac{\pi k}{n}$ angle. By hyperbolic trigonometry, we find that
$$\cosh\frac{d_k}{2}=\cosh d\sin\frac{\pi k}{n}.$$
In particular, if $k>k'$, $d_k>d_{k'}$ and two non-consecutive sides which are adjacent to the same side are closest to each other than any other two.

Consider now any simple closed geodesic on the surface. It cannot be contained in one of the two polygons, otherwise it would be contractible. So it needs to cross two sides. It cannot cross only two consecutive sides, otherwise it would be homotopic to a cusp. Hence it needs to cross two non-consecutive sides, so it contains at least two arcs of length at least $d_2$. Moreover, it is of length exactly $2d_2$ only if it is given by exactly two orthogonal between sides adjacent to the same side.

As such the curves we are considering are the only systoles.
\end{proof}

\bibliographystyle{alpha}
\bibliography{references}

\begin{thebibliography}{{Aou}15}

\bibitem[AH15]{ts}
Tarik Aougab and Shinnyih Huang.
\newblock Minimally intersecting filling pairs on surfaces.
\newblock {\em Algebr. Geom. Topol.}, 15(2):903--932, 2015.

\bibitem[{Aou}15]{aougab}
T.~{Aougab}.
\newblock {Local geometry of the k-curve graph}.
\newblock {\em ArXiv e-prints}, August 2015.

\bibitem[APP11]{app}
James~W. Anderson, Hugo Parlier, and Alexandra Pettet.
\newblock Small filling sets of curves on a surface.
\newblock {\em Topology Appl.}, 158(1):84--92, 2011.

\bibitem[FP14]{fp}
Federica {Fanoni} and Hugo {Parlier}.
\newblock {Systoles and kissing numbers of finite area hyperbolic surfaces}.
\newblock {\em To appear in Algebr. Geom. Topol.}, 2014.

\bibitem[JMM96]{jmm}
M.~Juvan, A.~Malni{\v{c}}, and B.~Mohar.
\newblock Systems of curves on surfaces.
\newblock {\em J. Combin. Theory Ser. B}, 68(1):7--22, 1996.

\bibitem[MRT14]{mrt}
Justin Malestein, Igor Rivin, and Louis Theran.
\newblock Topological designs.
\newblock {\em Geom. Dedicata}, 168:221--233, 2014.

\bibitem[Prz15]{przytycki}
Piotr Przytycki.
\newblock Arcs intersecting at most once.
\newblock {\em Geom. Funct. Anal.}, 25(2):658--670, 2015.

\bibitem[Sch94]{schmutz}
Paul Schmutz.
\newblock Congruence subgroups and maximal {R}iemann surfaces.
\newblock {\em J. Geom. Anal.}, 4(2):207--218, 1994.

\bibitem[Tro91a]{troyanov}
Marc Troyanov.
\newblock Prescribing curvature on compact surfaces with conical singularities.
\newblock {\em Trans. Amer. Math. Soc.}, 324(2):793--821, 1991.

\bibitem[Tro91b]{troyanov2}
Marc Troyanov.
\newblock The {S}chwarz lemma for nonpositively curved {R}iemannian surfaces.
\newblock {\em Manuscripta Math.}, 72(3):251--256, 1991.

\end{thebibliography}
\end{document}